\documentclass[11pt]{article}
\usepackage{amsmath,amsfonts,amssymb, amsthm,enumerate,graphicx}
\usepackage{tikz,authblk}
\usepackage{subfig}
\usepackage{url}
\newtheorem{theorem} {{\textsf{Theorem}}}
\newtheorem{proposition}[theorem]{{\textsf{Proposition}}}
\newtheorem{corollary}[theorem]{{\textsf{Corollary}}}

\newtheorem{definition}[theorem]{{\textsf{Definition}}}
\newtheorem{remark}[theorem]{{\textsf{Remark}}}
\newtheorem{example}[theorem]{{\textsf{Example}}}
\newtheorem{lemma}[theorem]{{\textsf{Lemma}}}

\newcommand{\TPSS}{\mathbb{S}^{\hspace{.2mm}1} \mbox{$\times
\hspace{-2.6mm}_{-}$} \, \mathbb{S}^{\hspace{.1mm}1}}
\newcommand{\TPSP}{\mathbb{S}^{\hspace{.2mm}2} \mbox{$\times
\hspace{-2.6mm}_{-}$} \, \mathbb{S}^{\hspace{.1mm}1}}

\begin{document}

\title{Minimal crystallizations of 3-manifolds with boundary}
\author{Biplab Basak$^1$ and Manisha Binjola}

\date{}

\maketitle

\vspace{-15mm}
\begin{center}

\noindent {\small Department of Mathematics, Indian Institute of Technology Delhi, New Delhi 110016, India.$^2$}


\footnotetext[1]{Corresponding author.}

\footnotetext[2]{{\em E-mail addresses:} \url{biplab@iitd.ac.in} (B.
Basak), \url{binjolamanisha@gmail.com} (M. Binjola).}

\medskip

\date{September 18, 2021}
\end{center}

\hrule

\begin{abstract}
 Let  $(\Gamma,\gamma)$ be a crystallization of connected compact 3-manifold $M$ with $h$ boundary components.  Let $\mathcal{G}(M)$ and $\mathit k (M)$ be the regular genus and gem-complexity of $M$ respectively,  and let $\mathcal{G}(\partial M)$ be the  regular genus of $\partial M$. We prove that 
 $$\mathit k (M)\geq 3 (\mathcal{G}(M)+h-1) \geq 3 (\mathcal{G} (\partial M)+h-1).$$
These bounds for gem-complexity of $M$ are sharp for several 3-manifolds with boundary. Further, we show that if $\partial M$ is connected  and $\mathit k (M)< 3 (\mathcal{G} (\partial M)+1)$ then $M$ is a handlebody. In particular, we prove that $\mathit k (M) =3 \mathcal{G} (\partial M)$   if $M$ is a handlebody and  $\mathit k (M) \geq 3 (\mathcal{G} (\partial M)+1)$ if $M$ is not a handlebody. Further, we obtain several combinatorial properties for a crystallization of 3-manifolds with boundary.
\end{abstract}

\noindent {\small {\em MSC 2020\,:} Primary 57Q15. Secondary 05C15; 57K30; 57K31; 57Q05.

\noindent {\em Keywords:} PL-manifolds,  Crystallizations, Regular genus, Gem-complexity, Handlebody.}

\medskip


\section{Introduction}

A crystallization $(\Gamma,\gamma)$ of a connected compact PL $d$-manifold (possibly with boundary) is a certain type of edge colored graph which represents the manifold (for details and related notations we refer Subsection \ref{crystal}). The journey of crystallization theory has begun due to Pezzana who gives the existence of a crystallization  for every closed connected PL $d$-manifold (see \cite{pe74}). Later the existence of a crystallization has been proved for every connected compact PL $d$-manifold with boundary (see \cite{cg80, ga83}). A beautiful proof of the classification of closed surfaces using crystallization theory can be found in \cite{bb17}. In \cite{ga79}, Gagliardi gave a combinatorial characterization of a $4$-colored graph to be a  crystallization of a closed connected $3$-manifold. In \cite{cg80} the authors extended the above result to a connected compact $3$-manifold with connected boundary, and in \cite{ga83}  Gagliardi further extended the result to a connected compact $3$-manifold with several boundary components. The lower bound for the number of vertices of a crystallization can be found in \cite{bd14} for closed connected $3$-manifolds, and in \cite{bc15} for closed connected $4$-manifolds. In this article, we gave a lower bound for the number of vertices of a crystallization for connected compact  $3$-manifolds with several boundary components. 

\smallskip

Extending the notion of genus in dimension 2, the notion of regular genus $\mathcal{G}(M)$ for a closed connected PL $d$-manifold $M$, has been introduced in \cite{ga81}, which is strictly related to the existence of regular
embeddings of crystallizations of the manifold into surfaces (cf. Subsection \ref{sec:genus} for details). Later, in \cite{ga87}, the concept of regular genus has been extended for a connected compact  PL $d$-manifold with boundary, for $d\geq 2$. The regular genus of a closed connected orientable (resp., a non-orientable) surface equals the genus (resp., half of the genus) of the surface. Several classification results according to  the gem-complexity and regular genus can be found in \cite{bb19,bb18, cav99, fg82}.  Let  $M$ be a  connected compact 3-manifold $M$ with boundary, and  let $\mathcal{G}(M)$ and $\mathcal{G}(\partial M)$ be the regular genera of $M$ and  $\partial M$ respectively. Then from \cite{bm87,cp90}, we know that $\mathcal{G}(M) \geq \mathcal{G} (\partial M)$. For $3\leq d \leq 5$, a classification result for a $d$-dimensional manifold with connected boundary can be found in \cite{ca93, ca92, ca90} when the regular genus of the manifold is same as the regular genus of its boundary. 

\smallskip

 The gem-complexity  is another interesting and useful combinatorial invariant for classifying connected compact PL $d$-manifolds $M$, and is defined as the non-negative integer $\mathit{k}(M) = p - 1$, where $2p$ is the minimum number of vertices of a crystallization of $M$.   A catalogue of closed connected 3-manifolds up to gem-complexity 14 can be found in \cite{bcg09,cc08}. A catalogue of PL 4-manifolds by gem-complexity can be found in \cite{cc15}. Such results for manifold with boundary are not very well known. The estimations of Matveev's complexity for 3-manifolds with boundary can be found in \cite{cc13}. In this article, we prove that if $M$ is a connected compact 3-manifold with $h$ boundary components then  $\mathit k (M)\geq 3 (\mathcal{G}(M)+h-1)$ (cf. Theorem \ref{theorem:gem-genus}). This bound is sharp for several $3$-manifolds with boundary (cf. Remark \ref{remark:sharp}). Since $\mathcal{G}(M) \geq \mathcal{G} (\partial M)$, we also have $\mathit k (M) \geq 3 (\mathcal{G} (\partial M)+h-1)$.  Further, we have shown that if $M$ is a $3$-manifold with connected boundary and $\mathit k (M)< 3 (\mathcal{G} (\partial M)+1)$ then $M$ is a handlebody (cf. Theorem \ref{theorem:upperbound}). In particular, we prove that  if $M$ is a handlebody then $\mathit k (M) =3 \mathcal{G} (\partial M)$ and  if $M$ is not a handlebody then $\mathit k (M) \geq 3 (\mathcal{G} (\partial M)+1)$.  We have shown the sharpness of this bound for several $3$-manifolds which are not handlebodies. (cf. Remark \ref{remark:sharp}).

 \smallskip

\section{Preliminaries}

\subsection{Crystallization} \label{crystal}

Crystallization theory is a combinatorial representation tool for piecewise-linear (PL) manifolds of arbitrary dimension. A multigraph is a graph where multiple edges are allowed but loops are forbidden. For a multigraph $\Gamma= (V(\Gamma),E(\Gamma))$, a surjective map $\gamma : E(\Gamma) \to \Delta_d:=\{0,1, \dots , d\}$ is called a proper edge-coloring if $\gamma(e) \ne \gamma(f)$ for any pair $e,f$ of adjacent edges. The elements of the set $\Delta_d$ are called the {\it colors} of $\Gamma$. A graph $(\Gamma,\gamma)$ is called {\it $(d+1)$-regular} if degree of each vertex is $d+1$ and is said to be {\it $(d+1)$-regular with respect to a color $c$} if after removing all the edges of color $c$ from $\Gamma$, the resulting graph is $d$-regular. We refer to \cite{bm08} for standard terminology on graphs. 

\smallskip

A graph $(\Gamma,\gamma)$ is called {\it $(d+1)$-regular colored graph} if $\Gamma$ is a $(d+1)$-regular and $\gamma$ is a proper edge-coloring.  A {\it $(d+1)$-colored graph with boundary} is a pair $(\Gamma,\gamma)$, where $\Gamma$ is a $(d+1)$-regular graph with respect to a color $c\in \Delta_d$ but not $(d+1)$-regular and $\gamma$ is a proper edge-coloring. If $(\Gamma,\gamma)$ is a $(d+1)$-regular colored graph or a $(d+1)$-colored graph with boundary then we simply call $(\Gamma,\gamma)$ as a $(d+1)$-colored graph. For each $B \subseteq \Delta_d$ with $h$ elements, the graph $\Gamma_B =(V(\Gamma), \gamma^{-1}(B))$ is an $h$-colored graph with edge-coloring $\gamma|_{\gamma^{-1}(B)}$. For a color set $\{i_1,i_2,\dots,i_k\} \subset \Delta_d$, $\Gamma_{\{i_1,i_2, \dots, i_k\}}$ denotes the subgraph restricted to the color set  $\{i_1,i_2,\dots,i_k\}$  and $g_{i_1i_2 \dots i_k}$ denotes the number of connected components of the graph $\Gamma_{\{i_1, i_2, \dots, i_k\}}$.  A graph $(\Gamma,\gamma)$ is called {\it contracted} if subgraph $\Gamma_{\hat{c}}:=\Gamma_{\Delta_d\setminus \{c\}}$ is connected for all $c$. 

\smallskip
 
Let $\mathbb{G}_d$ denote the set of graphs $(\Gamma,\gamma)$ which are $(d+1)$-regular with respect to the fixed color $d$. Thus  $\mathbb{G}_d$ contains all the $(d+1)$-regular colored graphs as well as all $(d+1)$-colored graphs with boundary. If $(\Gamma,\gamma)\in \mathbb{G}_d$ then the vertices with degree $d+1$ are called the internal vertices and the vertices with degree $d$ are called the boundary vertices. Let $C_{ij}$ denote the the number of $\{i,j\}$-colored cycles in $\Gamma$. Then $C_{ij}=g_{ij}$ for $i,j\in \Delta_d\setminus \{d\}$.   For each graph $(\Gamma,\gamma) \in \mathbb{G}_d$,  we define its  boundary graph $(\partial \Gamma,\partial \gamma)$ as follows:

\begin{itemize}
\item{} there is a bijection between $V(\partial \Gamma)$ and the set of boundary vertices of $\Gamma$;

\item{} $u_1,u_2 \in V(\partial \Gamma)$ are joined in $\partial \Gamma$ by an edge of color $j$ if and only if $u_1$ and $u_2$ are joined in $\Gamma$  by a path formed by $j$ and $d$ colored edges alternatively.
\end{itemize} 

\smallskip

Note that, if $(\Gamma,\gamma)$ is $(d+1)$-regular then $(\Gamma,\gamma)\in \mathbb{G}_d$ and $\partial \Gamma = \emptyset$. For each  $(\Gamma,\gamma) \in \mathbb{G}_d$, a corresponding $d$-dimensional simplicial cell-complex ${\mathcal K}(\Gamma)$ is determined as follows:

\begin{itemize}
\item{} for each vertex $u\in V(\Gamma)$, take a $d$-simplex $\sigma(u)$ and label its vertices by $\Delta_d$;

\item{} corresponding to each edge of color $j$ between $u,v\in V(\Gamma)$, identify the ($d-1$)-faces of $\sigma(u)$ and $\sigma(v)$ opposite to $j$-labeled vertices such that the vertices with same label coincide.
\end{itemize}

We refer to \cite{bj84} for CW-complexes and related notions. We say  $(\Gamma,\gamma)$ {\it represents} connected compact PL $d$-manifold $M$ (possibly with boundary) if the geometrical carrier $|{\mathcal K}(\Gamma)|$ is PL homeomorphic to $M$. It is not hard to see that  $|{\mathcal K}(\Gamma)|$ is orientable if and only if $\Gamma$ is a bipartite graph. If  $(\Gamma,\gamma)\in \mathbb{G}_d$ represents a connected compact  PL $d$-manifold with boundary then we can define its boundary graph $(\partial \Gamma,\partial \gamma)$, and each component of the boundary-graph $(\partial \Gamma,\partial \gamma)$ represents a component of $\partial M$.  From the construction it is easy to see that, for $\mathcal{B} \subset \Delta_d$ of cardinality $k+1$, ${\mathcal K}(\Gamma)$ has as many $k$-simplices with vertices labeled by $\mathcal{B}$ as many connected components of $\Gamma_{\Delta_d \setminus \mathcal{B}}$ are (cf. \cite{fgg86})

\smallskip

 Let $(\Gamma,\gamma)\in \mathbb{G}_d$ represent a connected compact PL  $d$-manifold with $h$ boundary components, then  $\mathcal{K}(\Gamma)$ has at least $d \cdot \max\{1,h\}+1$ vertices, for $h\geq 0$. For $h\geq 1$, it is easy to see that $\Gamma_{\hat d}$ is connected and each component of $\partial \Gamma$ is contracted if and only if $\mathcal{K}(\Gamma)$ has exactly $dh+1$ vertices.
 
\begin{definition}[\cite{ga83}]
Let  $(\Gamma,\gamma)  \in \mathbb{G}_d$ be a connected graph such that $\partial \Gamma$ has $h$ components, for $h\geq 1$. Then $(\Gamma,\gamma)$ is called $\partial$-contracted if $(a)$ $\Gamma_{\hat{d}}$ is connected, and $(b)$ for every $0\leq c \leq d-1$, $\Gamma_{\hat{c}}$ has $h$ components.
\end{definition}

A connected graph  $(\Gamma,\gamma) \in \mathbb{G}_d$ is said to be a {\em crystallization} of a connected compact PL $d$-manifold $M$ with $h$ boundary components if  $\mathcal{K}(\Gamma)$ has exactly  $d \cdot \max\{1,h\}+1$ vertices, for $h\geq 0$. In other words, a connected graph $(\Gamma,\gamma) \in \mathbb{G}_d$ is  a crystallization of a manifold $M$ with (non-empty) boundary if $(\Gamma,\gamma)$ is $\partial$-contracted. Note that, if $\partial M$ is connected (resp., empty) then $(\Gamma,\gamma)$ is contracted.

\smallskip

The starting point of the whole crystallization theory is the following Pezzana's Existence Theorem (cf. \cite{pe74}).

\begin{proposition}
 Every closed connected PL $d$-manifold admits a crystallization.
\end{proposition}

Pezzana's existence theorem has been extended to the boundary case (cf \cite{cg80, ga83}). 

\begin{proposition}[\cite{cg80, ga83}]
Let $M$ be a connected compact PL $d$-manifold with (possibly non-connected) boundary. For every crystallization $(\Gamma',\gamma')$ of $\partial M$, there exists a crystallization $(\Gamma,\gamma)$ of $M$, whose boundary graph $(\partial \Gamma,\partial \gamma)$ is isomorphic with $(\Gamma',\gamma')$.
\end{proposition}

It is known that a connected compact PL $d$-manifold with boundary can always be represented by a $(d+1)$-colored graph $(\Gamma,\gamma)$ which is regular with respect to a fixed color $k$, for some $k\in \Delta_{d}$. Without loss of generality, we can assume that $k=d$, i.e., $(\Gamma,\gamma) \in \mathbb{G}_d$. 

In \cite{ga79}, Gagliardi gave a combinatorial characterization of a contracted $4$-colored graph to be a  crystallization of a closed connected $3$-manifold. In \cite{ga83}, Gagliardi extended the above theorem to connected compact $3$-manifold with several boundary components. For graph $(\Gamma,\gamma)$ with boundary,  let $2p$ and $2\overline{p}$ denote the number of vertices and boundary vertices of $(\Gamma,\gamma)$ respectively.

\begin{proposition}[\cite{ga83}]\label{prop:boundary}
A  $4$-colored graph with boundary $(\Gamma,\gamma)$ is a crystallization of a connected compact $3$-manifold $M$ with $h$ boundary components ($h\geq 1$) if and only if the following conditions hold.
\begin{enumerate}[$(i)$]
\item $(\Gamma,\gamma)$ is connected, $\partial$-contracted element of $\mathbb{G}_3$, and $\partial \Gamma$ has $h$ components. 
\item $g_{03}-g_{12}=g_{13}-g_{02}=g_{23}-g_{01}=\frac{\bar p}{2}+\frac{h}{2} -1.$
\item $g_{0 1}+ g_{02}+g_{12}=2+p.$
\end{enumerate}
\end{proposition}

A new invariant `gem-complexity' has been defined. Given a connected compact  PL $d$-manifold $M$, its {\em gem-complexity} is the non-negative integer $\mathit{k}(M) = p - 1$, where $2p$ is the minimum number of vertices of a crystallization of $M$.

Let $(\Gamma_1,\gamma_1)$ and $(\Gamma_2,\gamma_2)$ be two disjoint $(d+1)$-colored graphs (possibly with boundary) with the same color set $\Delta_d$, and let $v_i \in V(\Gamma_i)$ ($1\leq i\leq 2$). The {\em connected sum} $(\Gamma_1\#_{v_1v_2}\Gamma_2,\gamma_1\#_{v_1v_2}\gamma_2)$ of $\Gamma_1$, $\Gamma_2$  with respect to vertices $v_1, v_2$  is the graph obtained from $(\Gamma_1 \setminus\{v_1\}) \sqcup (\Gamma_2 \setminus \{v_2\})$ by adding $d+1$ new edges $e_0, \dots, e_d$ with colors $0, \dots, d$ respectively, such that the end points of $e_j$ are $u_{j,1}$ and $u_{j, 2}$, where $v_i$ and $u_{j,i}$ are joined in $(\Gamma_i,\gamma_i)$ with an edge of color $j$ for $0\leq j\leq d$, $1\leq i\leq 2$. From the construction it is clear that $\mathcal K (\Gamma_1\#_{v_1v_2}\Gamma_2)$ is obtained from $\mathcal K (\Gamma_1)$ and $\mathcal K (\Gamma_2)$ by removing the $d$-simplices $\sigma(v_1)$ and $\sigma(v_2)$ and pasting together all $(\Delta_d \setminus \{j\})$-colored  $(d-1)$-simplices of $\sigma(u_{j,1})$ and $\sigma(u_{j,2})$ for $0\leq j \leq d$. For $0\leq j \leq d$, if the $j$-colored vertex of $\sigma(v_1)$ or $\sigma(v_2)$ is not a boundary vertex of $\mathcal K (\Gamma_1)$ or $\mathcal K (\Gamma_2)$ respectively then $(\Gamma_1\#_{v_1v_2}\Gamma_2,\gamma_1\#_{v_1v_2}\gamma_2)$ represents $|\mathcal K (\Gamma_1)| \# |\mathcal K (\Gamma_2)|$. 

\subsection{Regular Genus of PL $d$-manifolds (possibly with boundary)}\label{sec:genus}
Let $(\Gamma,\gamma) \in \mathbb{G}_d$ be a $(d+1)$-colored graph which represents a connected compact $d$-manifold $M$ (possibly with boundary $\partial M$). For each boundary vertex $u$ (possibly empty), a new vertex $u'$ and a new $d$-colored edge is added between $u$ and $u'$. In this way a new graph $(\Gamma',\gamma')$ is obtained. If $\Gamma$ has no boundary vertex then $\Gamma'$ is same as $\Gamma$.  Let $2p$ and $2\overline{p}$ denote the number of vertices and boundary vertices of $(\Gamma,\gamma)$ respectively. Then the number of interior vertices is $2\dot{p}:=2p-2\overline{p}$

Now, given any cyclic permutation $\varepsilon = (\varepsilon _0,\varepsilon _1,\dots,\varepsilon_d = d)$ of the color set $\Delta_d$, a regular imbedding of $\Gamma'$ into a surface $F$ is simply an imbedding $i:|\Gamma'|\to F$ such that the vertices of $i(\Gamma')\cap\partial F$  are the images of the new added vertices and the regions of the imbedding are  bounded by either a cycle (internal region)  or by a walk (boundary region) of $\Gamma'$ with $\varepsilon_i,\varepsilon_{i+1}(i$ mod $d+1)$ colored edges alternatively. 

Using Gross `voltage theory' (see \cite{gro74}), in the bipartite case, and Stahl `embedding schemes' (see \cite{st78}), in the non bipartite case, one can prove that
 for every cyclic permutation $\varepsilon$ of $\Delta_d$, a regular embedding $i_\varepsilon : \Gamma' \hookrightarrow F_\varepsilon$ exists, where orientable (resp., non-orientable) surface $F_\varepsilon$ is of Euler characteristic 
\begin{eqnarray}\label{relation_chi}
\chi_{\varepsilon}(\Gamma)= \sum_{i \in \mathbb{Z}_{d+1}}C_{\varepsilon_i \varepsilon_{i+1}} + (1-d) \dot{p} + (2-d)  \overline{p}
\end{eqnarray}
and $\lambda_\varepsilon=\partial g_{\varepsilon_0 \varepsilon_{d-1}}$ holes where $\partial g_{ij}$ denotes the number of $\{i,j\}$-colored cycles of $\partial\Gamma$. For more details we refer \cite{bcg01, ga87}.

\smallskip

In the orientable (resp., non-orientable) case, the integer
$$\rho_{\varepsilon}(\Gamma) = 1 - \chi_{\varepsilon}(\Gamma)/2-\lambda_\varepsilon/2$$
is equal to the genus (resp,. half of the genus) of the surface $F_{\varepsilon}$.
Then, the \textit{regular genus of $(\Gamma,\gamma)$} denoted by $\rho (\Gamma)$ and the \textit{regular genus of $M$} denoted by $\mathcal G (M)$ are defined as follows (cf. \cite{ga81}):

$$\rho(\Gamma)= \min \{\rho_{\varepsilon}(\Gamma) \ | \  \varepsilon =(\varepsilon_0,\varepsilon_1,\dots,\varepsilon_d = d)\ \text{ is a cyclic permutation of } \ \Delta_d\};$$
$$\mathcal G(M) = \min \{\rho(\Gamma) \ | \  (\Gamma,\gamma) \mbox{ represents } M\}.$$

In dimension two, it is easy to see that if $(\Gamma,\gamma)$ represents a surface $F$, then the corresponding $(\Gamma',\gamma')$ regularly imbeds into $F$ itself. Hence, for each surface $F$,

\begin{align*}
\mathcal G(F)=
\left\{ \begin{array}{lcl}
genus(F)  & \mbox{if} & F \mbox{ is orientable}, \\
\frac{1}{2} \times genus(F) & \mbox{if} & F \mbox{ is non-orientable}.
\end{array}\right. \nonumber
\end{align*}
If $M$ is a connected compact  $3$-manifold with $h$ boundary components say $\partial^1M,\dots, \partial^hM$, then $\mathcal G(\partial M)$ is defined as $\sum_{i=1}^h \mathcal G(\partial^i M)$. Further, from \cite{ca93, ca92, ca90}, we have the following result.

\begin{proposition}\label{prop:handlebody}
For $3\leq d \leq 5$, let $M$ be a $d$-dimensional (orientable or nonorientable) manifold with connected boundary $\partial M$. Then the regular genus of $M$ is equal to the regular genus of the boundary $\partial M$ (say, $g$) if and only if  $M$ is a $d$-dimensional genus $g$ handlebody.
\end{proposition}

\subsection{3-dimensional handlebodies}
Let $M$ be a connected compact  3-manifold with boundary. Then each component of the boundary $\partial M$ of $M$ is a closed connected surface.
%
%
A handlebody can be defined as the simplest 3-manifold with connected boundary - in the sense that it contains pairwise disjoint, properly embedded 2-discs such that the manifold resulting from cutting along the discs is a 3-ball. Up to homeomorphism, there are exactly two handlebodies (one is orientable and another is non-orientable) of any positive integer genus.

\section{Main results}
Let  $(\Gamma,\gamma)$ be a 4-colored graph with the color set $\Delta_3=\{0,1,2,3\}$ regular with respect to the color $3$, i.e., $(\Gamma,\gamma)\in \mathbb{G}_3$.  Let $2p$ and $2\overline{p}$ denote the number of vertices and boundary vertices of $(\Gamma,\gamma)$ respectively. Thus the number of interior vertices is $2p-2\overline{p}$ and number of $3$-colored edges is  $p-\overline{p}$.   Let $g_{ij},C_{ij}$ denote the number of components and cycles in $\Gamma_{ij}$ for $0\leq i,j\leq 3$.
Then it is easy to see that $g_{i3}=\bar{p}+C_{i3}$,  for $0\leq i \leq 2$.

Let $M$ be a connected compact $3$-manifold with connected boundary $\partial M$. If $\partial M$ is a  non-orientable surface then $\partial M$ must be  an $n$-connected sum of Klein bottles, for some $n\geq 1$. Thus, the regular genus of the boundary surface $\partial M$ is a non-negative integer. 

\begin{lemma}\label{lemma:boundarygenus}
For  $n\in \mathbb{N}\cup\{0\}$,  let $(\Gamma,\gamma)$ be a crystallization of a connected compact $3$-manifold $M$ with connected  boundary surface of regular genus $n$. Then  $|V(\partial \Gamma)| = 2+4n$.
 \end{lemma}

\begin{proof}
Since  $(\Gamma,\gamma)$ is a crystallization of a connected compact  $3$-manifold $M$ with connected  boundary,
$(\partial \Gamma,\partial \gamma)$ is a crystallization of $\partial M$. Because $\partial M$ has regular genus $n$, we conclude that $\chi(\mathcal{K}(\partial \Gamma))=2-2n$.
Let $V, E, F$ be the number of the vertices, edges and triangles of the corresponding simplicial cell complex $\mathcal{K}(\partial \Gamma)$ respectively. Then $V= 3$ and $2E=3F$. Thus,
$$V-E+F=\chi(\partial M)=2-2n.$$
This implies, $F=2+4n$, i.e., $|V(\partial\Gamma)| = 2+4n$. 
\end{proof}

 \begin{lemma}\label{lemma:g-C}
 Let $(\Gamma,\gamma)\in \mathbb{G}_3$ be a crystallization of a  connected compact $3$-manifold with connected boundary.  Let $|V(\partial\Gamma)|=2+4n$, for some $n\in \mathbb{N}\cup\{0\}$.
 Then for $\{i,j,k\}=\{0,1,2\}$,
 \begin{enumerate}[$(i)$]
 \item $g_{i3} = 1+2n +C_{i3}$,
  \item $g_{i3} = n +g_{jk}$,
  \item $g_{jk}=1+n+C_{i3}$.
 \end{enumerate}  
 \end{lemma}
 
 \begin{proof}
Since $2\overline{p}=2+4n$ and $g_{i3}=\bar{p}+C_{i3}$, it follows  that
$$
g_{i3}=1+2n+C_{i3}.
$$
Further, it follows from Proposition \ref{prop:boundary} that
\begin{align*}
g_{i3}-g_{jk}&=\frac{\overline{p}}{2}-\frac{1}{2}\\
&=n.
\end{align*}
These prove Parts $(i)$ and $(ii)$, and Part $(iii)$ follows from Parts  $(i)$ and $(ii)$.
 \end{proof}

From Proposition \ref{prop:boundary}, we have $g_{01}+g_{02}+g_{12}=2+p$. Therefore, it is easy to prove the following.

\begin{corollary}\label{cor:bound}
Let $(\Gamma,\gamma)\in \mathbb{G}_3$ be a crystallization of a connected compact $3$-manifold  with connected  boundary surface of regular genus $n$. 
Then $|V(\Gamma)|\geq 2+6n$ and
\begin{align*}
g_{03}, g_{13}, g_{23}&\geq 1+2n, \\
g_{01}, g_{02}, g_{12}&\geq 1+n.
\end{align*}
Moreover, if  $|V(\Gamma)| = 2+6n$ then
\begin{align*}
g_{03}, g_{13}, g_{23}&= 1+2n, \\
g_{01}, g_{02}, g_{12}&=1+n.
\end{align*}
\end{corollary}

%
%

\begin{lemma}\label{lemma:genus}
For  $n\in \mathbb{N}\cup\{0\}$,  let $(\Gamma,\gamma)\in \mathbb{G}_3$ be a crystallization of a connected compact  $3$-manifold with connected boundary of regular genus $n$. Then $\Gamma$ represents a handlebody if and only if at least one of $g_{01}$, $g_{02}$ and $g_{12}$ attains the minimum (i.e., equals to $1+n$).
\end{lemma}

\begin{proof}
Let $2\overline{p}$ and $2\dot{p}$ denote the number of boundary vertices and internal vertices respectively, and $2p=2\overline{p}+2\dot{p}$. Since $(\Gamma,\gamma)$ is a crystallization of a connected compact  $3$-manifold $M$ with connected boundary surface of regular genus $n$, we have $2\bar p = 2+4n$ by Lemma \ref{lemma:boundarygenus}. 
Moreover, Corollary  \ref{cor:bound} implies, $2p\geq 2+6n$. Let $2p=2r+6n$, for $r\geq 1$. Because $2\overline{p}=2+4n$, we have $2\dot{p}=2r+2n-2$. Further, it follows from Lemma \ref{lemma:g-C} that
 $g_{12}-C_{03}=g_{02}-C_{13}=g_{01}-C_{23}=n+1$. Let $C_{23}=a, C_{13}=b$ and $C_{03}=c$. Then
 
$$\begin{matrix}
g_{01}&g_{02}&g_{12}&C_{03}&C_{13}&C_{23}\\
(1+n+a)&(1+n+b)&(1+n+c)&c&b&a
\end{matrix}$$
Since $g_{01}+g_{02}+g_{12}=p+2=r+3n+2$, it follows that $r=a+b+c+1$.
Thus, for $\epsilon=(\epsilon_0,\epsilon_1,\epsilon_2,\epsilon_3=3)$, we have
\begin{align*}
\chi_\epsilon (\Gamma)&= g_{\epsilon_0 \epsilon_1}+g_{\epsilon_1 \epsilon_2}+C_{\epsilon_2 \epsilon_3}+C_{\epsilon_0 \epsilon_3}-(2r+2n-2)-(1+2n)\\
&=g_{\epsilon_0 \epsilon_1}+g_{\epsilon_1 \epsilon_2}+C_{\epsilon_2 \epsilon_3}+C_{\epsilon_0 \epsilon_3}-1-4n-2(a+b+c).
\end{align*}
Thus, for different permutations $\epsilon=(\epsilon_0,\epsilon_1,\epsilon_2,\epsilon_3=3)$, the possible choices for $\chi_\epsilon(\Gamma)$ are $1-2n-2a$, $1-2n-2b$ or $1-2n-2c$.
Since $(\partial \Gamma,\partial \gamma)$  is a crystallization of $\partial M$, $\partial g_{ij}=1$, i.e., $\lambda_\epsilon=1$. Thus, the possible choices for $\rho_\epsilon (\Gamma)$ are $n+a$, $n+b$ or $n+c$.
Therefore,
%
 $\rho(\Gamma)=\min\{n+a,n+b,n+c\}$. From Proposition \ref{prop:handlebody}, we know that
$\Gamma$ represents a handlebody if and only if $\rho(\Gamma)=n$. Thus, $\Gamma$ represents a handlebody if and only if at least one of $g_{01}$, $g_{02}$ and $g_{12}$ equals to $(1+n)$.
\end{proof}

\begin{lemma}\label{lemma:gem-genus}
Let $M$ be a connected compact $3$-manifold with connected boundary. Then, $\mathit{k} (M) \geq 3 \mathcal{G}(M)$. 
\end{lemma}

\begin{proof}
Let $(\Gamma,\gamma) \in \mathbb{G}_3$ be a crystallization for $M$ with $2p$ vertices. Let $\mathcal{G}(\partial M)=n$. Then, from the proof of Lemma \ref{lemma:genus}, we have the regular genus of $(\Gamma,\gamma)$, $\rho(\Gamma)=  \min  \{n+a,n+b,n+c\}$, where $C_{23}=a, C_{13}=b$ and $C_{03}=c$. It follows from Lemma \ref{lemma:g-C},  $g_{jk}=1+n+C_{i3}$ for $i,j,k \in \{0,1,2\}$. Since $p=(g_{01}+g_{02}+g_{12})-2$,  $p= 1+3n+a+b+c$. Thus, $p-1= (n+a)+(n+b)+(n+c) \geq 3 \rho(\Gamma) \geq 3\mathcal{G}(M)$. Since $(\Gamma,\gamma) \in \mathbb{G}_3 $ is an arbitrary crystallization for $M$, we have $\mathit{k} (M) \geq 3 \mathcal{G}(M)$. 
\end{proof}

\begin{example}[Gem-complexity of a handlebody $M$ is  $3\mathcal{G}(\partial M)$]{\rm Let $M$ be a 3-dimensional handlebody. It is known that $\partial M$ is either $\mathbb{S}^2$, $\#_n (\mathbb{S}^1\times \mathbb{S}^1 )$ or $\#_n (\TPSS)$ for some $n\geq 1$. For $n\in \mathbb{N}\cup\{0\}$, we have given a crystallization of a connected compact `$3$-manifold $M$ with boundary surface $\#_n (\mathbb{S}^1\times \mathbb{S}^1 )$ (cf. Part $(a)$ of Figure \ref{fig:1}) and $\#_n (\TPSS)$ (cf. Part $(b)$ of Figure \ref{fig:1})' with exactly $2+6n$ vertices, where $n=\mathcal{G}(\partial M)$. Since $\rho(\Gamma)=\mathcal{G}(\partial M)$, $M$ is a handlebody. Therefore, by Lemma \ref{lemma:gem-genus} we have the gem-complexity of a handlebody $M$ is  $3\mathcal{G}(\partial M)$.}
\end{example}

\begin{figure}[ht]
\tikzstyle{vert}=[circle, draw, fill=black!100, inner sep=0pt, minimum width=4pt] \tikzstyle{vertex}=[circle,
draw, fill=black!00, inner sep=0pt, minimum width=4pt] \tikzstyle{ver}=[] \tikzstyle{extra}=[circle, draw,
fill=black!50, inner sep=0pt, minimum width=2pt] \tikzstyle{edge} = [draw,thick,-] \centering
\begin{tikzpicture}[scale=0.8]

\begin{scope}[shift={(5,7)}]
\node[ver] (3) at (1,-4.5){\tiny{$3$}}; 
\node[ver] (2) at (1,-4){\tiny{$2$}};
\node[ver](1) at (1,-3.5){\tiny{$1$}}; 
\node[ver](0) at (1,-3){\tiny{$0$}}; 
\node[ver] (8) at (3,-4.5){}; 
\node[ver](7) at (3,-4){}; 
\node[ver](6) at (3,-3.5){}; 
\node[ver] (5) at (3,-3){};
\path[edge] (0) -- (5);
\draw [line width=3pt, line cap=round, dash pattern=on 0pt off 2\pgflinewidth] (0) -- (5);
\path[edge] (1) -- (6);
\path[edge,dotted] (2) -- (7);
\path[edge,dashed] (3) -- (8);

\end{scope}

\begin{scope}[shift={(0,8)}]
\foreach \x/\y/\z/\w in {-3/2/2.5/v_{6},-3/-2/-2.5/v_{7},1/2/2.5/v_{12},1/-2/-2.5/v_{13},4/1/1.5/v_{6n},4/-1/-1.5/v_{6n+1}}
{\node[ver](\w) at (\x,\z){\tiny{$\w$}};
\node[vert] (\w) at (\x,\y){};} 

\foreach \x/\y/\z/\w in {-1/2/2.5/v_{8},3/2/2.5/v_{6n-4},-4/1/1.5/v_{2}}
{\node[ver](\w) at (\x,\z){\tiny{$\w$}};
\node[vertex] (\w) at (\x,\y){};}

\node[ver]() at (5.5,0){\tiny{$v_{0}$}};
\node[vertex] (v_{0}) at (5,0){};
\node[ver]() at (-5.5,0){\tiny{$v_{1}$}};
\node[vert] (v_{1}) at (-5,0){};
\node[ver]() at (-4.3,-1.2){\tiny{$v_{3}$}};
\node[vertex] (v_{3}) at (-4,-1){};

\node[ver]() at (-0.7,-2.2){\tiny{$v_{9}$}};
\node[vertex] (v_{9}) at (-1,-2){};

\node[ver]() at (3.5,-2.2){\tiny{$v_{6n-3}$}};
\node[vertex] (v_{6n-3}) at (3,-2){};

\foreach \x/\y/\z/\w in {-4/-3/-3.5/v_{4},-1/-3/-0.5/v_{10},3/-3/3.5/v_{6n-2}}
{\node[ver](\w) at (\z,\y){\tiny{$\w$}};
\node[vert] (\w) at (\x,\y){};}

\foreach \x/\y/\z/\w in {-4/-4/-3.5/v_{5},-1/-4/-0.5/v_{11},3/-4/3.5/v_{6n-1}}
{\node[ver](\w) at (\z,\y){\tiny{$\w$}};
\node[vertex] (\w) at (\x,\y){};}

\foreach \x/\y in {1.7/2,2/2,2.3/2,1.7/-2,2/-2,2.3/-2}
{\node[extra] () at (\x,\y){};}

\foreach \x/\y in {v_{1}/v_{2},v_{1}/v_{3},v_{2}/v_{6},v_{3}/v_{7},v_{7}/v_{9},v_{6}/v_{8},v_{8}/v_{12},v_{9}/v_{13},v_{6n-4}/v_{6n},v_{6n}/v_{0},v_{0}/v_{6n+1},v_{6n-3}/v_{6n+1}}
{\path[edge] (\x) -- (\y);}

\foreach \x/\y in {v_{1}/v_{2},v_{3}/v_{7},v_{6}/v_{8},v_{9}/v_{13},v_{6n}/v_{0},v_{6n-3}/v_{6n+1}}
{\draw [line width=3pt, line cap=round, dash pattern=on 0pt off 2\pgflinewidth]  (\x) -- (\y);}

\draw[edge] plot [smooth,tension=1] coordinates{(v_{4}) (-4.2,-3.5) (v_{5})};
\draw[edge] plot [smooth,tension=1] coordinates{(v_{4}) (-3.8,-3.5) (v_{5})};
\draw[line width=3pt, line cap=round, dash pattern=on 0pt off 2\pgflinewidth] plot [smooth,tension=1] coordinates{(v_{4}) (-3.8,-3.5) (v_{5})};

\draw[edge] plot [smooth,tension=1] coordinates{(v_{10}) (-1.2,-3.5) (v_{11})};
\draw[edge] plot [smooth,tension=1] coordinates{(v_{10}) (-0.8,-3.5) (v_{11})};
\draw[line width=3pt, line cap=round, dash pattern=on 0pt off 2\pgflinewidth] plot [smooth,tension=1] coordinates{(v_{10}) (-0.8,-3.5) (v_{11})};

\draw[edge] plot [smooth,tension=1] coordinates{(v_{6n-2}) (3.2,-3.5) (v_{6n-1})};
\draw[edge] plot [smooth,tension=1] coordinates{(v_{6n-2}) (2.8,-3.5) (v_{6n-1})};
\draw[line width=3pt, line cap=round, dash pattern=on 0pt off 2\pgflinewidth] plot [smooth,tension=1] coordinates{(v_{6n-2}) (3.2,-3.5) (v_{6n-1})};

\foreach \x/\y in {v_{3}/v_{4},v_{9}/v_{10},v_{6n-3}/v_{6n-2}}
{\path[edge, dashed] (\x) -- (\y);}

\path[edge, dotted] (v_{0}) -- (v_{1});

\foreach \x/\y/\z/\w in {v_{2}/v_{4}/-4.7/-1.5,v_{6}/v_{5}/-2.5/-1.5,v_{3}/v_{7}/-3.2/-1.2, v_{8}/v_{10}/-1.5/-1.5, v_{9}/v_{13}/0/-1.5,v_{12}/v_{11}/1.7/-1.7,v_{6n+1}/v_{6n-3}/3.2/-1.2, v_{6n-4}/v_{6n-2}/2.5/-1.5, v_{6n}/v_{6n-1}/4.5/-2}{
\path[edge, dotted] plot [smooth,tension=1] coordinates{(\x) (\z,\w) (\y)};}

\node[ver]() at (0,-5){$(a)$ \mbox{A crystallization of the orientable handlebody.}};
\end{scope}

\begin{scope}[shift={(0,0)}]
\foreach \x/\y/\z/\w in {-3/2/2.5/v_{6},-3/-2/-2.5/v_{7},1/2/2.5/v_{12},1/-2/-2.5/v_{13},4/1/1.5/v_{6n},4/-1/-1.5/v_{6n+1}}
{\node[ver](\w) at (\x,\z){\tiny{$\w$}};
\node[vert] (\w) at (\x,\y){};} 

\foreach \x/\y/\z/\w in {-1/2/2.5/v_{8},3/2/2.5/v_{6n-4},-4/1/1.5/v_{2}}
{\node[ver](\w) at (\x,\z){\tiny{$\w$}};
\node[vertex] (\w) at (\x,\y){};}

\node[ver]() at (5.5,0){\tiny{$v_{0}$}};
\node[vertex] (v_{0}) at (5,0){};
\node[ver]() at (-5.5,0){\tiny{$v_{1}$}};
\node[vert] (v_{1}) at (-5,0){};
\node[ver]() at (-4.3,-1.2){\tiny{$v_{3}$}};
\node[vertex] (v_{3}) at (-4,-1){};

\node[ver]() at (-0.7,-2.2){\tiny{$v_{9}$}};
\node[vertex] (v_{9}) at (-1,-2){};

\node[ver]() at (3.5,-2.2){\tiny{$v_{6n-3}$}};
\node[vertex] (v_{6n-3}) at (3,-2){};

\foreach \x/\y/\z/\w in {-4/-3/-3.5/v_{4},-1/-3/-0.5/v_{10},3/-3/3.5/v_{6n-2}}
{\node[ver](\w) at (\z,\y){\tiny{$\w$}};
\node[vert] (\w) at (\x,\y){};}

\foreach \x/\y/\z/\w in {-4/-4/-3.5/v_{5},-1/-4/-0.5/v_{11},3/-4/3.5/v_{6n-1}}
{\node[ver](\w) at (\z,\y){\tiny{$\w$}};
\node[vertex] (\w) at (\x,\y){};}

\foreach \x/\y in {1.7/2,2/2,2.3/2,1.7/-2,2/-2,2.3/-2}
{\node[extra] () at (\x,\y){};}

\foreach \x/\y in {v_{1}/v_{2},v_{1}/v_{3},v_{2}/v_{6},v_{3}/v_{7},v_{7}/v_{9},v_{6}/v_{8},v_{8}/v_{12},v_{9}/v_{13},v_{6n-4}/v_{6n},v_{6n}/v_{0},v_{0}/v_{6n+1},v_{6n-3}/v_{6n+1}}
{\path[edge] (\x) -- (\y);}

\foreach \x/\y in {v_{1}/v_{2},v_{3}/v_{7},v_{6}/v_{8},v_{9}/v_{13},v_{6n}/v_{0},v_{6n-3}/v_{6n+1}}
{\draw [line width=3pt, line cap=round, dash pattern=on 0pt off 2\pgflinewidth]  (\x) -- (\y);}

\draw[edge] plot [smooth,tension=1] coordinates{(v_{4}) (-4.2,-3.5) (v_{5})};
\draw[edge] plot [smooth,tension=1] coordinates{(v_{4}) (-3.8,-3.5) (v_{5})};
\draw[line width=3pt, line cap=round, dash pattern=on 0pt off 2\pgflinewidth] plot [smooth,tension=1] coordinates{(v_{4}) (-3.8,-3.5) (v_{5})};

\draw[edge] plot [smooth,tension=1] coordinates{(v_{10}) (-1.2,-3.5) (v_{11})};
\draw[edge] plot [smooth,tension=1] coordinates{(v_{10}) (-0.8,-3.5) (v_{11})};
\draw[line width=3pt, line cap=round, dash pattern=on 0pt off 2\pgflinewidth] plot [smooth,tension=1] coordinates{(v_{10}) (-0.8,-3.5) (v_{11})};

\draw[edge] plot [smooth,tension=1] coordinates{(v_{6n-2}) (3.2,-3.5) (v_{6n-1})};
\draw[edge] plot [smooth,tension=1] coordinates{(v_{6n-2}) (2.8,-3.5) (v_{6n-1})};
\draw[line width=3pt, line cap=round, dash pattern=on 0pt off 2\pgflinewidth] plot [smooth,tension=1] coordinates{(v_{6n-2}) (3.2,-3.5) (v_{6n-1})};

\foreach \x/\y in {v_{3}/v_{4},v_{9}/v_{10},v_{6n-3}/v_{6n-2}}
{\path[edge, dashed] (\x) -- (\y);}

\path[edge, dotted] (v_{0}) -- (v_{1});

\foreach \x/\y/\z/\w in {v_{2}/v_{5}/-4.7/-1.5,v_{6}/v_{4}/-3/-0.5,v_{3}/v_{7}/-3.2/-1.2, v_{8}/v_{10}/-1.5/-1.5, v_{9}/v_{13}/0/-1.5,v_{12}/v_{11}/1.7/-1.7,v_{6n+1}/v_{6n-3}/3.2/-1.2, v_{6n-4}/v_{6n-2}/2.5/-1.5, v_{6n}/v_{6n-1}/4.5/-2}{
\path[edge, dotted] plot [smooth,tension=1] coordinates{(\x) (\z,\w) (\y)};}

\node[ver]() at (0,-5){$(b)$ \mbox{A crystallization of the non-orientable handlebody.}};
\end{scope}

\end{tikzpicture}
\caption{Crystallizations of the orientable and non-orientable handlebodies with $6n+2$ vertices.}\label{fig:1}
\end{figure}

\begin{theorem} \label{theorem:gem-genus}
Let $M$ be a connected compact $3$-manifold with $h$ boundary components. Then, $\mathit{k} (M) \geq 3 (\mathcal{G}(M)+h-1)$. 
\end{theorem}

\begin{proof}
Let $(\Gamma,\gamma) \in \mathbb{G}_3$ be a crystallization for $M$ with $2p$ vertices. Let $2\overline{p}$ and $2\dot{p}$ be the number of boundary vertices and interior vertices of $(\Gamma,\gamma)$ respectively.  Let $(\Gamma', \gamma')$ be a graph obtained from $(\Gamma,\gamma)$ by adding $h-1$ number of edges of color $3$ 
between vertices belonging to different boundary components to connect all of them. Then
$(\Gamma', \gamma')$ is a crystallization for $|\mathcal{K}(\Gamma')|$, a $3$-manifold with connected boundary and $|V(\Gamma')|=|V(\Gamma)|=2p$. Let $2\dot{p}^ \prime$, $2\overline{p}^\prime$ be the number of interior and boundary vertices of $(\Gamma', \gamma')$ respectively. For $i,j \in \{0,1,2,3 \}$, let $g_{ij}'$ and $C_{ij}'$ denote the number of components and cycles in $\Gamma_{ij}'$ respectively. Let $\partial g_{ij}^ \prime$ be the number of $\{i,j\}$-colored cycles in $\partial \Gamma ^ \prime$. Then $2\dot{p}'=2\dot{p}+2(h-1)$, $2\overline{p}'= 2\overline{p} -2(h-1)$, $C_{i3}'=C_{i3}$, $g_{ij}'= g_{ij}$ and $\partial g_{ij}'=\partial g_{ij}-(h-1)$, for $0 \leq i,j \leq 2$. Therefore, for any permutation $\varepsilon=(\varepsilon_0, \varepsilon_1,\dots, \varepsilon_3 =3)$ of $\Delta_3$, $\chi_\varepsilon (\Gamma)= \chi_\varepsilon(\Gamma^\prime)+(h-1)$. Then, $\rho_\varepsilon(\Gamma)= \rho_\varepsilon (\Gamma^\prime)-(h-1)$. Since $\varepsilon=(\varepsilon_0, \varepsilon_1,\dots, \varepsilon_3 =3)$ is an arbitrary permutation of $\Delta_3$, $\rho (\Gamma^ \prime)=\rho(\Gamma)+h-1$. Now, by the proof of Lemma \ref{lemma:gem-genus}, we have $p-1\geq 3 \rho(\Gamma')=3(\rho(\Gamma)+h-1)\geq 3(\mathcal{G}(M)+h-1)$. Since $(\Gamma,\gamma) \in \mathbb{G}_3$ is arbitrary crystallization for $M$, $\mathit{k} (M) \geq 3 (\mathcal{G}(M)+h-1)$. 
\end{proof}

Let $M$ be a connected compact $3$-manifold with  boundary. Then from \cite{bm87,cp90}, we know that $\mathcal{G}(M) \geq \mathcal{G} (\partial M)$. Thus we have the following result.

\begin{corollary}\label{cor:gem-boundary} 
Let $M$ be a connected compact $3$-manifold with $h$ boundary components. Then, $\mathit{k} (M) \geq 3 (\mathcal{G}(M)+h-1) \geq 3 (\mathcal{G}(\partial M)+h-1)$. 
\end{corollary}

The above result suggests that  a crystallization $(\Gamma,\gamma)$ of a connected compact $3$ manifold with connected boundary has at least  $6\mathcal{G}(\partial M)+2$ vertices. Further, it is easy to see that if $\mathcal{G}(\partial M)=0$ then $|V(\Gamma)|<8$  implies $M$ is the 3-ball $D^3$, which we consider as a trivial handlebody. It is also easy to construct an 8-vertex crystallization of a connected compact  $3$-manifold $M$ with spherical boundary such that $M$ is not $D^3$. Now, we state and prove a similar result when  $M$ is a connected compact $3$-manifold with non-spherical boundary.

\begin{theorem}\label{theorem:upperbound}
If $M$ is a $3$-manifold with connected boundary and $\mathit k (M)< 3 (\mathcal{G} (\partial M)+1)$ then $M$ is a handlebody.
\end{theorem}

\begin{proof}
Let $(\Gamma,\gamma)\in \mathbb{G}_3$ be a crystallization of the $3$-manifold $M$ such that $2p=|V(\Gamma)| < 8 +6\mathcal{G}(\partial M)$. From Corollary \ref{cor:gem-boundary},  we know that $2p \geq 6\mathcal{G}(\partial M)+2$ and we have a crystallization of a `handlebody $M$ with boundary surface $\#_n (\mathbb{S}^1\times \mathbb{S}^1 )$ or $\#_n (\TPSS)$' with exactly $2+6\mathcal{G}(\partial M)$ vertices. Now, we claim that if $2p=2+6\mathcal{G}(\partial M), 4+6\mathcal{G}(\partial M)$ or $6+6\mathcal{G}(\partial M)$, then  $M$ is a handlebody.

Let $n$ be the regular genus of the boundary surface $\partial M$. Then Lemma \ref{lemma:boundarygenus} implies, $2\overline{p}= 2+4n$.  It follows from Lemma \ref{lemma:g-C},  $g_{jk}=1+n+C_{i3}$ for $i,j,k \in \{0,1,2\}$. If $C_{i3} \geq 1$ for all $i\in \{0,1,2\}$ then $2p=2g_{01}+2g_{02}+2g_{12}-4\geq 8+6n=8+6\mathcal{G}(\partial M)$, which is a contradiction. Thus, at least one of $g_{01}$, $g_{02}$ and $g_{12}$  equals to $1+n$. It follows from Lemma \ref{lemma:genus} that  $M$ is a handlebody.
\end{proof}

\begin{corollary}\label{cor:gem-complexity}
Let $M$ be a connected compact $3$-manifold with connected boundary. Then
\begin{enumerate}[$(a)$]
\item If $M$ is a handlebody then $\mathit{k}(M) = 3\mathcal{G}(\partial M)$.
\item If $M$ is not a handlebody then $\mathit{k}(M) \geq 3(\mathcal{G}(\partial M)+1)$.
\end{enumerate}
\end{corollary}

\begin{remark} \label{remark:eq}
{\rm Let $(\Gamma_1,\gamma_1)$ be a crystallization of the orientable or non-orientable handlebody $M_1$ with $2+6\mathcal{G}(\partial M)$ vertices as constructed in Figure \ref{fig:1}. Let $(\Gamma_2,\gamma_2)$ be the unique 8-vertex crystallization of $\mathbb{S}^2\times \mathbb{S}^1$ or $\TPSP$ or $\mathbb{RP}^3$ (cf. \cite[Figure 2]{bd14}). Let $v_1\in V(\Gamma_1)$ be an interior vertex and $v_2\in V(\Gamma_2)$. Then the graph connected sum $(\Gamma,\gamma)=(\Gamma_1\#_{v_1v_2}\Gamma_2,\gamma_1\#_{v_1v_2}\gamma_2)$ is a crystallization of a connected compact  $3$-manifold $M$ with boundary such that $M$ is not a handlebody and $\partial M=\partial M_1$. Here $|V(\Gamma)|=6\mathcal{G}(\partial M)+2+8-2=6\mathcal{G}(\partial M)+8$. Thus, $\mathit{k}(M) =3(\mathcal{G}(\partial M)+1)$.
}
 \end{remark}

A $1$-dipole of color $j \in \Delta_d$ of a $(d+1)$-colored graph (possibly with boundary) $(\Gamma,\gamma) \in \mathbb{G}_d$ is a subgraph $\theta$ of $\Gamma$ consisting of two vertices $x,y$ joined by color $j$ such that 
$\Gamma_{\hat{j}} (x) \neq \Gamma_ {\hat{j}} (y)$, where $\Gamma_ {\hat{j}} (u)$ denotes the component of $\Gamma_ {\hat{j}}$ containing $u$. The cancellation of $1$-dipole from $\Gamma$ consists of two steps: first deleting $\theta$ from $\Gamma$ and second welding the same colored hanging edges (see \cite{fg82c} for more details).

\begin{corollary}\label{cor:gen-h-bdry}
Let $M$ be a $3$-manifold with $h$ boundary components such that $\mathit k (M)< 3 (\mathcal{G} (\partial M)+h)$. Let $(\Gamma,\gamma)\in \mathbb{G}_3$ be a crystallization of $M$ such that $\mathit k (M)\leq |V(\Gamma)|/2 -1 < 3 (\mathcal{G} (\partial M)+h)$. Then the new graph   $(\Gamma^\prime,\gamma^\prime) \in \mathbb{G}_3$ after cancelling all possible $1$-dipoles from $(\Gamma,\gamma)$, represents a handlebody.
\end{corollary}

\begin{proof}
Let $(\Gamma,\gamma)\in \mathbb{G}_3$ be a crystallization of $M$ such that 
$|V(\Gamma)|=2p < 6\mathcal{G}(\partial M)+6h+2$. It follows from Corollary \ref{cor:gem-boundary} that $2p \geq 6\mathcal{G}(\partial M)+6h-4$.
Let $(\Gamma^\prime,\gamma^\prime)\in \mathbb{G}_3 $ be the new crystallization after cancelling $(h-1)$ number of $1$-dipoles of  color $j$ from $(\Gamma,\gamma)$, for $0\leq j \leq 2$. Then, $|V(\Gamma^\prime)|=2p-6(h-1)$.
Let $M^\prime$ denote the manifold with connected boundary with crystallization $(\Gamma^\prime,\gamma^\prime)$. Then $\mathcal{G} (\partial M)=\mathcal{G} (\partial M^\prime)$.
Further,  $$6\mathcal{G} (\partial M)+6h-4 \leq 2p < 6\mathcal{G} (\partial M)+6h+2$$
$$\Rightarrow 6\mathcal{G} (\partial (M^\prime))+2 \leq |V(\Gamma^\prime)|< 6\mathcal{G} (\partial M)+8.$$
Then Theorem \ref{theorem:upperbound} implies $M^\prime$ is a handlebody.
\end{proof}

\begin{remark} 
{\rm
Let $M$ be a $3$-manifold with $h$ boundary components as in Corollary \ref{cor:gen-h-bdry}. Then $M$ need not be the connected sum of handlebodies. For an example, we have a crystallization for $3$-manifold $\mathbb{RP}^2\times [0,1]$ in Figure \ref{fig:3} which satisfies the hypothesis of Corollary \ref{cor:gen-h-bdry} but is not a connected sum of handlebodies.}
\end{remark}

\begin{corollary}
Let  $S$ be a closed connected surface. Then $$6 \mathcal{G}(S)+3 \leq \mathit{k}(S\times [0,1]) \leq 8 \mathcal{G}(S)+3.$$
\end{corollary}

\begin{proof}
Let  $(\bar \Gamma,\bar \gamma)$ be a crystallization of $M=S\times [0,1]$. Here $\partial M$ has exactly two components. Therefore, by Theorem \ref{theorem:gem-genus}, we have $\mathit{k}(S\times [0,1]) \geq 3(\mathcal{G}(M)+1) \geq  6\mathcal{G}(S)+3$. On the other hand, it is easy to construct a crystallization of $M=S\times [0,1]$ with $16 \, \mathcal{G}(S) +8$ vertices by the following procedure:

Let  $(\Gamma,\gamma)$ be a crystallization of $S$ with color set $\Delta_2=\{0,1,2\}$. Let $2p$ be the number of vertices of $\Gamma$. Then $2p=4\mathcal{G}(S)+2$. Let $v_1,v_2,\dots,v_{2p}$ be the vertices of $\Gamma$. Now, choose a fixed order of the colors say, $(0,1,2)$. For $1\leq m\leq 4$, let $G_m(i,j,k)$ be the graph obtained from $(\Gamma,\gamma)$ by replacing vertices $v_l$ by $v_l^{(m)}$ and by replacing the triplet of colors $(0,1,2)$ by $(i,j,k)$, where $1\leq l \leq 2p$ and $0\leq i \neq j\neq k \leq 3$. Now consider the four 2-colored graphs $G_1(0,1,-)$, $G_2(-,1,3)$, $G_3(2,-,3)$ and $G_4(2,0,-)$, where by `$-$', we mean the corresponding color is missing. Let $(\bar \Gamma,\bar \gamma)$ be a graph obtained by 
$(i)$ adding $2p$ edges of color $2$ between $v_l^{(1)}$ of $G_1(0,1,-)$ and $v_l^{(2)}$ of $G_2(-,1,3)$, for $1\leq l \leq 2p$, $(ii)$  adding $2p$ edges of color $0$ between $v_l^{(2)}$ of $G_2(0,1,-)$ and $v_l^{(3)}$ of $G_3(-,1,3)$, for $1\leq l \leq 2p$, and  $(iii)$  adding $2p$ edges of color $1$ between $v_l^{(3)}$ of $G_3(0,1,-)$ and $v_l^{(4)}$ of $G_4(-,1,3)$, for $1\leq l \leq 2p$.

Then $(\bar \Gamma,\bar \gamma)$ is a crystallization (cf. Figure \ref{fig:3} for $S=\mathbb{RP}^2$) of $S\times [0,1]$ with $8p=16 \, \mathcal{G}(S) +8$ vertices.  Therefore, $\mathit{k}(S\times [0,1]) \leq 8 \, \mathcal{G}(S) +3$.
\end{proof}

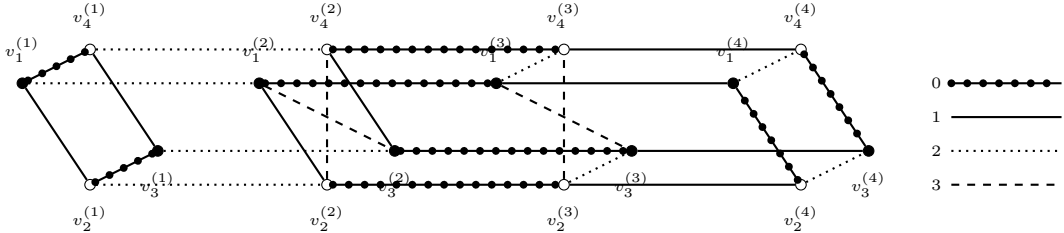
\begin{figure}[ht]
\tikzstyle{vert}=[circle, draw, fill=black!100, inner sep=0pt, minimum width=4pt] \tikzstyle{vertex}=[circle,
draw, fill=black!00, inner sep=0pt, minimum width=4pt] \tikzstyle{ver}=[] \tikzstyle{extra}=[circle, draw,
fill=black!50, inner sep=0pt, minimum width=2pt] \tikzstyle{edge} = [draw,thick,-] \centering
\begin{tikzpicture}[scale=0.9]

\begin{scope}[shift={(12.5,3.5)}]
\node[ver] (3) at (1,-4.5){\tiny{$3$}}; 
\node[ver] (2) at (1,-4){\tiny{$2$}};
\node[ver](1) at (1,-3.5){\tiny{$1$}}; 
\node[ver](0) at (1,-3){\tiny{$0$}}; 
\node[ver] (8) at (3,-4.5){}; 
\node[ver](7) at (3,-4){}; 
\node[ver](6) at (3,-3.5){}; 
\node[ver] (5) at (3,-3){};
\path[edge] (0) -- (5);
\draw [line width=3pt, line cap=round, dash pattern=on 0pt off 2\pgflinewidth] (0) -- (5);
\path[edge] (1) -- (6);
\path[edge,dotted] (2) -- (7);
\path[edge,dashed] (3) -- (8);

\end{scope}

\begin{scope}[shift={(0,0)}]
\foreach \x/\y/\z/\w/\u in {0/0.5/0/1/v_1^{(1)}, 2/-0.5/2/-1/v_3^{(1)}}
{\node[ver]() at (\z,\w){\tiny{$\u$}};
\node[vert] (\u) at (\x,\y){};} 

\foreach \x/\y/\z/\w/\u  in {1/-1/1/-1.5/v_2^{(1)}, 1/1/1/1.5/v_4^{(1)}}
{\node[ver]() at (\z,\w){\tiny{$\u$}};
\node[vertex] (\u) at (\x,\y){};} 
\end{scope}

\begin{scope}[shift={(3.5,0)}]
\foreach \x/\y/\z/\w/\u in {0/0.5/0/1/v_1^{(2)}, 2/-0.5/2/-1/v_3^{(2)}}
{\node[ver]() at (\z,\w){\tiny{$\u$}};
\node[vert] (\u) at (\x,\y){};} 

\foreach \x/\y/\z/\w/\u  in {1/-1/1/-1.5/v_2^{(2)}, 1/1/1/1.5/v_4^{(2)}}
{\node[ver]() at (\z,\w){\tiny{$\u$}};
\node[vertex] (\u) at (\x,\y){};} 
\end{scope}

\begin{scope}[shift={(7,0)}]
\foreach \x/\y/\z/\w/\u in {0/0.5/0/1/v_1^{(3)}, 2/-0.5/2/-1/v_3^{(3)}}
{\node[ver]() at (\z,\w){\tiny{$\u$}};
\node[vert] (\u) at (\x,\y){};} 

\foreach \x/\y/\z/\w/\u  in {1/-1/1/-1.5/v_2^{(3)}, 1/1/1/1.5/v_4^{(3)}}
{\node[ver]() at (\z,\w){\tiny{$\u$}};
\node[vertex] (\u) at (\x,\y){};} 
\end{scope}

\begin{scope}[shift={(10.5,0)}]
\foreach \x/\y/\z/\w/\u in {0/0.5/0/1/v_1^{(4)}, 2/-0.5/2/-1/v_3^{(4)}}
{\node[ver]() at (\z,\w){\tiny{$\u$}};
\node[vert] (\u) at (\x,\y){};} 

\foreach \x/\y/\z/\w/\u  in {1/-1/1/-1.5/v_2^{(4)}, 1/1/1/1.5/v_4^{(4)}}
{\node[ver]() at (\z,\w){\tiny{$\u$}};
\node[vertex] (\u) at (\x,\y){};} 
\end{scope}

\foreach \x/\y in {v_1^{(1)}/v_2^{(1)},v_2^{(1)}/v_3^{(1)},v_3^{(1)}/v_4^{(1)},v_4^{(1)}/v_1^{(1)},v_4^{(4)}/v_3^{(4)},v_2^{(4)}/v_1^{(4)},v_4^{(2)}/v_4^{(3)},v_3^{(2)}/v_3^{(3)},v_2^{(2)}/v_2^{(3)},v_1^{(2)}/v_1^{(3)},v_4^{(4)}/v_4^{(3)},v_3^{(4)}/v_3^{(3)},v_2^{(4)}/v_2^{(3)},v_1^{(4)}/v_1^{(3)},v_1^{(2)}/v_2^{(2)},v_4^{(2)}/v_3^{(2)}}
{\path[edge] (\x) -- (\y);}

\foreach \x/\y in {v_2^{(1)}/v_3^{(1)},v_4^{(1)}/v_1^{(1)},v_4^{(4)}/v_3^{(4)},v_2^{(4)}/v_1^{(4)},v_4^{(2)}/v_4^{(3)},v_3^{(2)}/v_3^{(3)},v_2^{(2)}/v_2^{(3)},v_1^{(2)}/v_1^{(3)}}
{\draw[line width=3pt, line cap=round, dash pattern=on 0pt off 2\pgflinewidth]  (\x) -- (\y);}

\foreach \x/\y in {v_2^{(2)}/v_4^{(2)},v_1^{(2)}/v_3^{(2)},v_2^{(3)}/v_4^{(3)},v_1^{(3)}/v_3^{(3)}}
{\path[edge,dashed] (\x) -- (\y);}

\foreach \x/\y in {v_1^{(1)}/v_1^{(2)},v_2^{(1)}/v_2^{(2)},v_3^{(1)}/v_3^{(2)},v_4^{(1)}/v_4^{(2)},v_1^{(3)}/v_4^{(3)},v_2^{(3)}/v_3^{(3)},v_1^{(4)}/v_4^{(4)},v_2^{(4)}/v_3^{(4)}}
{\path[edge, dotted] (\x) -- (\y);}

\end{tikzpicture}
\caption{A crystallization of $\mathbb{RP}^2\times [0,1]$.}\label{fig:3}
\end{figure}

\begin{remark} \label{remark:sharp}
{\rm Let $\mathbb{D}^3$ denote the $3$-ball, and $\mathbb{H}^3_g, \tilde {\mathbb{H}}^3_g$ denote the orientable and non-orientable handlebody of genus $g$. Then our bounds $\mathit k (M) \geq 3 (\mathcal{G} (M)+h-1)$ and $\mathit k (M) \geq 3 (\mathcal{G} (\partial M)+h-1)$ both are sharp for the compact $3$-manifolds $(\#_{h_1} \mathbb{H}^3_{g_i})\#(\#_{h_2} \tilde{\mathbb{H}}^3_{\tilde{g}_j})$, where $h_1,h_2\geq 0$, $h_1+h_2=h$, $g_i,\tilde{g}_j \geq 0$. Further, our bounds $\mathit k (M) \geq 3 (\mathcal{G} (M)+h-1)$ is sharp for the compact $3$-manifolds $M\#(\#_h \mathbb{D}^3)$, $M\#(\#_h \mathbb{H}^3_{g_i}), M\#(\#_h \tilde{\mathbb{H}}^3_{\tilde{g}_j})$ where $M=\mathbb{RP}^3,\mathbb{S}^2\times\mathbb{S}^1,\TPSP$. In Corollary \ref{cor:gem-complexity}, we prove that  if $M$ has connected boundary and is not a handlebody then $\mathit{k}(M) \geq 3(\mathcal{G}(\partial M)+1)$. This bound is also sharp for $M\#\mathbb{D}^3$, $M\#\mathbb{H}^3_{g}, M\# \tilde{\mathbb{H}}^3_{g}$ where $M=\mathbb{RP}^3,\mathbb{S}^2\times\mathbb{S}^1,\TPSP$.
}
 \end{remark}

\bigskip

\noindent {\bf Acknowledgement:} 
The first author is supported by DST INSPIRE Faculty Research Grant (DST/INSPIRE/04/2017/002471).

{\footnotesize

\end{document}